\documentclass[11pt,a4paper,reqno]{amsart} 
\usepackage{amssymb}
\usepackage{latexsym}
\usepackage{amsmath}
\usepackage{mathrsfs}
\usepackage{color}
\usepackage{calc}                         
\usepackage{cancel}

\newcommand{\scal}[2]{\langle #1,#2\rangle}

\newcommand{\rr}[1]{\mathbf R^{#1}}
\newcommand{\nn}[1]{\mathbf N^{#1}}

\newcommand{\cc}[1]{\mathbf C^{#1}}

\newcommand{\nm}[2]{\Vert #1\Vert _{#2}}

\newcommand{\nmm}[1]{\Vert #1\Vert }

\newcommand{\sets}[2]{\{ \, #1\, ;\, #2\, \} }

\newcommand{\cdo}{\, \cdot \, }

\newcommand{\vrum}{\vspace{0.1cm}}

\newcommand{\maclA}{\mathcal A}

\newcommand{\maclB}{\mathcal B}
\newcommand{\maclH}{\mathcal H}

\newcommand{\mascD}{\mathscr D}

\newcommand{\mascS}{\mathscr S}

\newcommand{\RE}{\operatorname{Re}}

\numberwithin{equation}{section}          

\newtheorem{thm}{Theorem}
\numberwithin{thm}{section}
\newtheorem{prop}[thm]{Proposition}

\newtheorem{lemma}[thm]{Lemma}

\theoremstyle{definition}

\theoremstyle{remark}

\newtheorem{rem}[thm]{Remark}              

\title{Non-isometric translation and modulation 
invariant Hilbert spaces}

\author{P. K. Ratnakumar}

\address{Harish-Chandra Research Institute (HBNI), Chhatnag Road,
Jhunsi, Allahabad, 211019, Uttarpradesh, India}

\email{ratnapk@hri.res.in}

\author{Joachim Toft}

\address{Department of Mathematics,
Linn{\ae}us University, V{\"a}xj{\"o}, Sweden}

\email{joachim.toft@lnu.se}

\author{Jasson Vindas}

\address{Department of MathematicsDepartment of Mathematics: Analysis,
Logic and Discrete Mathematics,
Ghent University, Ghent, Belgium}

\email{jasson.vindas@ugent.be}

\begin{document}

\begin{abstract}
Let $\maclH$ be a Hilbert space, continuously
embedded in $\mascS '(\rr d)$, and
which contains at least one non-zero element
in $\mascS '(\rr d)$.
If there is a constant $C_0>0$ such that
$$
\nm {e^{i\scal \cdo \xi}f(\cdo -x)}{\maclH}\le C_0\nm f{\maclH},
\qquad f\in \maclH ,\ x,\xi \in \rr d,
$$
then we prove that $\maclH = L^2(\rr d)$, with equivalent norms.
\end{abstract}

\keywords{modulation spaces, Feichtinger's minimization principle}

\subjclass{46C15, 46C05, 42B35}

\maketitle

\section{Introduction}\label{sec0}

\par

In \cite{ToGuMaRa} it is proved that a suitable non-trivial
Hilbert space $\maclH \subseteq \mascS '(\rr d)$, which is
norm preserved under translations
\begin{align*}
f &\mapsto f(\cdo -x) 
\intertext{and modulations}
f &\mapsto e^{i\scal \cdo \xi }f 
\end{align*}
is equal to $L^2(\rr d)$. (See \cite{Ho1} or Section
\ref{sec1} for notations.) It is here also proved that the
norms between $\maclH$ and $L^2(\rr d)$ only differs
by a multiplicative constant, i.{\,}e. for some
constant $C>0$ one has
\begin{equation}\label{Eq:NormEquiv}
\nm f{\maclH}
=
C\nm f{L^2(\rr d)},
\qquad f\in \maclH =L^2(\rr d).
\end{equation}
This property is analogous to Feichtinger's minimization property,
%
%
which shows that the Feichtinger algebra $S_0(\rr d)$, which is the same as the
modulation space $M^{1,1}(\rr d)$, is the smallest non-trivial Banach space
of tempered distributions which is
norm invariant under translations and modulations.
(See e.{\,}g. \cite{Gc1} for general facts about modulation spaces.)

\par

The condition on non-triviality
in \cite{ToGuMaRa} is simply that $\maclH$ should
contain at least one non-trivial element in $M^{1,1}(\rr d)$.
Hence the main result in \cite{ToGuMaRa}
can be formulated as follows.

\par

\begin{thm}\label{Thm:InvHilbertSpaces0}
Let $\maclH$ be a Hilbert space which is
continuously embedded in $\mascS '(\rr d)$,
and norm preserved under translations and
modulations. If $\maclH$ contains a non-zero element
in $M^{1,1}(\rr d)$, then $\maclH = L^2(\rr d)$, and
\eqref{Eq:NormEquiv}
holds for some constant $C>0$ which is independent
of $f\in \maclH = L^2(\rr d)$.
\end{thm}

\par

An alternative approach, using the Bargmann transform,
to reach similar properties is given by Bais, Pinlodi and
Venku Naidu in \cite{BaMoVe}.
In fact, by using their result \cite[Theorem 3.1]{BaMoVe} in combination
with some well-known arguments in the distribution theory, one obtains
the following improvement of Theorem \ref{Thm:InvHilbertSpaces0}.

\par

\begin{thm}\label{Thm:InvHilbertSpaces1}
Let $\maclH$ be a Hilbert space which is
continuously embedded in $\mascS '(\rr d)$,
and norm preserved under translations and
modulations. If $\maclH$ is non-trivial,
then $\maclH = L^2(\rr d)$, and
\eqref{Eq:NormEquiv}
holds for some constant $C>0$ which is independent
of $f\in \maclH = L^2(\rr d)$.
\end{thm}

\par

We observe that the condition

\begin{equation}\label{Eq:TypeCondGaussian}
%
%
|(f,\phi )|\le C\nm f{\maclH},
\qquad
f\in \maclH ,\ \phi (x)= e^{-\frac 12|x|^2},
\end{equation}
for some constant $C>0$ in the hypothesis in
\cite[Theorem 3.1]{BaMoVe} is absent in Theorem
\ref{Thm:InvHilbertSpaces1}. Hence Theorem
\ref{Thm:InvHilbertSpaces1} is slightly more general
than \cite[Theorem 3.1]{BaMoVe}.

\par

Here we remark that we may relax the hypothesis in Theorem \ref{Thm:InvHilbertSpaces1}
by assuming that $\maclH$ is continuously embedded in the space $\mascD '(\rr d)$
instead of the smaller space $\mascS '(\rr d)$.
In fact, let $\maclB$ be a translation invariant Banach space
which is continuously embedded in $\mascD '(\rr d)$ and satisfies
$$
\nm {f(\cdo -x)}{\maclB}\le C\nm f{\maclB},
$$ 
for some constant $C\ge 1$ which is independent of $f\in \maclB$
and $x\in \rr d$. Then it follows by some standard arguments that
$\maclB$ is continuously embedded in $\mascS '(\rr d)$
(see e.{\,}g. \cite[Proposition 1.5]{ToGuMaRa}).


In the paper we investigate properties on weaker forms of translation and modulation
invariant Hilbert spaces compared to \cite{BaMoVe,ToGuMaRa}. More precisely we
show that except for the norm identity \eqref{Eq:NormEquiv},
Theorem \ref{Thm:InvHilbertSpaces1} still holds true after the
condition that $\maclH$ is norm preserved under translations and modulations,
is relaxed into the weaker condition
\begin{equation}\label{Eq:ContTransMod}
\nm {f(\cdo -x)e^{i\scal \cdo \xi}}{\maclH} \le C_0\nm f{\maclH},
\qquad f\in \maclH ,\ x,\xi \in \rr d.
\end{equation}
In our extension, the identity \eqref{Eq:NormEquiv} should be replaced by
the norm equivalence
\begin{equation}\label{Eq:NormEquiv2}
C_0^{-1}C\nm f{L^2(\rr d)}
\le
\nm f{\maclH}
\le
C_0C\nm f{L^2(\rr d)},
\qquad f\in \maclH =L^2(\rr d).
\end{equation}

\par

More precisely our extension of Theorem \ref{Thm:InvHilbertSpaces1} is
the following.

\par

\begin{thm}\label{Thm:InvHilbertSpaces2}
Let $\maclH$ be a Hilbert space which is
continuously embedded in $\mascS '(\rr d)$,
and such that \eqref{Eq:ContTransMod} holds true for some constant
$C_0\ge 1$ which is independent of $f\in \maclH$ and $x,\xi \in \rr d$.
If $\maclH$ is non-trivial, then $\maclH = L^2(\rr d)$, and
\eqref{Eq:NormEquiv2} holds
for some constant $C>0$ which is independent
of $f\in \maclH = L^2(\rr d)$.
\end{thm}

\par

The key step for the proof of Theorem
\ref{Thm:InvHilbertSpaces2} is to
find an equivalent Hilbert norm
to $\nm \cdo{\maclH}$, which is
norm preserved under translations
and modulations. The result
then follows from Theorem 
\ref{Thm:InvHilbertSpaces1}. 
As a first idea one may try to use
the equivalent norm
$$
\nm f{\maclB}\equiv \sup _{x,\xi \in \rr d}
\left (
\nm {f(\cdo -x)e^{i\scal \cdo \xi}}{\maclH}
\right ).
$$
A straight-forward control shows that
this norm is invariant under translations
and modulations. On the other hand,
it seems $\nm f{\maclB}$ might fail to be
a Hilbert norm, and thereby not being
suitable for applying Theorem
\ref{Thm:InvHilbertSpaces1}.

\par

In our approach to find the sought Hilbert
norm, we use some ideas in the
construction of $\nm \cdo{\maclB}$
above, but replace the supremum with
mean-values of the form
$$
\nm f{[R]}^2
\equiv
(2R)^{-2d}\iint _{\rr {2d}}
\nm {f(\cdo -x)e^{i\scal \cdo \xi}}{\maclH}^2
\, dxd\xi ,
$$
when $R>0$. It follows that each
$\nm \cdo{[R]}$ is a Hilbert norm
which is uniformly equivalent with
$\nm \cdo{\maclH}$. On the other hand,
none of $\nm \cdo{[R]}$ need to be
translation nor modulation invariant.
However, by increasing $R$, it follows that
$\nm \cdo{[R]}$
becomes, in some sense, closer to being
translation and modulation invariant.
From suitable limit process, letting $R$
tending to infinity, we are able to
extract a sought equivalent Hilbert norm
which is norm preserved under translations
and modulations. (See Lemma
\ref{Lemma:ConvEquivHilbNorm}.)

\par

In Section \ref{sec3} we also present some improvements of this result,
where it is assumed that $\maclH$ in Theorem \ref{Thm:InvHilbertSpaces2}
is embedded in suitable larger (ultra-)distribution spaces which
contain $\mascS '(\rr d)$ (see Theorem \ref{Thm:InvHilbertSpaces3}).
These investigations are based on some
general properties of translation and modulation invariant quasi-Banach spaces
(see Proposition \ref{Prop:TransModInvQBanachSp}).

\par

The proofs of Theorem \ref{Thm:InvHilbertSpaces2} and its extensions,
are, among others, based on the fact that the 
involved Hilbert spaces
are separable.
In Section
\ref{sec2} we verify such facts, by using the Bargmann transform to
transfer the questions to Hilbert spaces
of \emph{entire functions}. In the end we conclude that any Hilbert space
which are continuously embedded in $\mascS '(\rr d)$ (or in some even
larger distribution spaces), must be separable. (See Propositions
\ref{Prop:Separability} and \ref{Prop:Separability2}.)
%
%
%

\par

\section*{Acknowledgement}
The first and second authors are grateful to HRI and the
infosys foundations for support for the second
 author during his visit at HRI in December 2023, making this collaboration
possible. The second 
 author is also grateful for the excellent hospitality
during his visit at HRI in December 2023.
The second 
 author was supported by Vetenskapsr{\aa}det (Swedish
Science Council) within the project 2019-04890.
The third author was supported by the Research
Foundation–Flanders through the FWO-grant number G067621N.

%
%
%
%

\par

%

\section{Proof of the main result}\label{sec1}

\par

In this section we prove our main result Theorem
\ref{Thm:InvHilbertSpaces2}. First we show how the condition
\eqref{Eq:TypeCondGaussian} in \cite[Theorem 3.1]{BaMoVe} can be
removed, which leads to Theorem \ref{Thm:InvHilbertSpaces1}. Then
we prove a lemma, which in combination with Theorem \ref{Thm:InvHilbertSpaces1},
essentially leads to Theorem \ref{Thm:InvHilbertSpaces2}.

\par

First we have the following lemma.
%
%

\par

\begin{lemma}\label{Lemma:NonGaussAssump}
Suppose that $\maclH$ 
is continuously embedded in $\mascS' (\rr d)$.
Then there is a constant $C>0$ such that \eqref{Eq:TypeCondGaussian}
holds.
\end{lemma}

\par

Lemma \ref{Lemma:NonGaussAssump} follows by a straight-forward
continuity argument. (See e.{\,}g.
the arguments in \cite[p.35]{Rud}.)
In order to assist the reader we present a
proof in Section \ref{sec2}
for a more general result (see Lemma
\ref{Lemma:NonGaussAssump2} and its proof).

\par

\begin{proof}[Proof of Theorem \ref{Thm:InvHilbertSpaces1}]
The result follows by combining \cite[Theorem 3.1]{BaMoVe}
with Lemma \ref{Lemma:NonGaussAssump}.
The details are left for the reader.
\end{proof}

\par

The proof of Theorem \ref{Thm:InvHilbertSpaces2} is based on the
fact that the involved Hilbert spaces are separable. This fact is
guaranteed by the following proposition.

\par

\begin{prop}\label{Prop:Separability}
Suppose that the Hilbert space $\maclH$ 
is continuously embedded in $\mascS' (\rr d)$.
Then $\maclH$ is separable.
\end{prop}

\par

We observe that neither translation nor modulation
invariant hypotheses are imposed on the Hilbert spaces
in Proposition \ref{Prop:Separability}.

\par

It is expected that Proposition \ref{Prop:Separability}
is available in the literature. For completeness
we give a proof of a generalized result
in Section \ref{sec2} (see Proposition \ref{Prop:Separability2}
and its proof).

\par

\begin{lemma}\label{Lemma:ConvEquivHilbNorm}
Suppose that $\maclH$ satisfies the hypothesis in
Theorem \ref{Thm:InvHilbertSpaces2}.
%
%
%
Then there is a norm $\nmm \cdo$ on $\maclH$ with the following properties:
\begin{enumerate}
\item $\nmm \cdo$ is equivalent to $\nm \cdo{\maclH}$;

\vrum

\item $\nmm \cdo$ is a Hilbert norm;

\vrum

\item $\nmm {e^{i\scal \cdo \xi}f(\cdo -x)}= \nmm f$ for every $f\in \maclH$
and $x,\xi \in \rr d$.
\end{enumerate}
\end{lemma}

\par

\begin{proof}
%
%
Let $Q_R = [-R,R]^{2d}$ be the cube in $\rr {2d}$ with center at origin and
side length $2R$, $R\ge 1$. Then the volume of $Q_R$ is given by
$|Q_R|=(2R)^{2d}$. We define the norm $\nm \cdo{[R]}$ by the formula
$$
\nm f{[R]} ^2
=
\frac 1{|Q_R|}\iint _{Q_R} \nm {e^{i\scal \cdo \xi}f(\cdo -x)}{\maclH}^2\, dxd\xi ,
\qquad f\in \maclH .
$$
By the definition it follows that $\nm \cdo {[R]}$ is equivalent to
$\nm \cdo {\maclH}$, and \eqref{Eq:ContTransMod} gives
\begin{equation}\label{Eq:NormEquiv1}
C_0^{-1}\nm f{\maclH} \le \nm f{[R]}
\le
C_0\nm f{\maclH},
\qquad f\in \maclH .
\end{equation}
Furthermore, $\nm \cdo{[R]}$ is a norm which arises from the scalar product
\begin{equation*}
(f,g)_{[R]} 
=
\frac 1{|Q_R|}\iint _{Q_R}
({e^{i\scal \cdo \xi}f(\cdo -x)},{e^{i\scal \cdo \xi}g(\cdo -x)})_{\maclH}
\, dxd\xi ,
\quad
f,g\in \maclH .
\end{equation*}

\par

Now let $\mho _0$ be a dense countable subset
of $\maclH$ containing $0$, and let $\mho$ be the smallest set
which contains $\mho _0$, and is closed under multiplications
by $\pm i$, additions and subtractions.
%
Note that $\mho _0$ and $\mho$ exist, because
$\maclH$ is separable, due to Proposition \ref{Prop:Separability}.
Then $\mho$ is countable.
Let $\mho = \{ f_j \} _{j=1}^\infty$
be a counting of $\mho$. Since
$$
C_0^{-1}\nm {f_1}{\maclH} \le \nm {f_1}{[m]}\le C_0\nm {f_1}{\maclH},
$$
for every $m\in \mathbf N$, there is a subsequence
$\{ m_k \} _{k=1}^\infty$ of $\mathbf N$
such that
$$
\nmm {f_1} \equiv \lim _{k\to \infty}  \nm {f_1}{[m _k]}
$$
exists. By Cantor's diagonalization principle,
there is a subsequence $\{ n_k \} _{k=1}^\infty$ of  $\{ m_k \} _{k=1}^\infty$
such that
$$
\nmm {f_j} \equiv \lim _{k\to \infty}  \nm {f_j}{[n_k]}
$$
exists for every $f_j\in \mho$, and by \eqref{Eq:NormEquiv1}
we get
\begin{equation}\label{Eq:NormEquiv3}
C_0^{-1}\nm f{\maclH} \le \nmm f
\le
C_0\nm f{\maclH},
\end{equation}
when $f\in \mho$.

\par

Next suppose that $f\in \maclH$ is arbitrary,
and let $\{ f_{0,j}\} _{j=1}^\infty
\subseteq \mho _0$, be chosen such that
$$
\lim _{j\to \infty} \nm {f-f_{0,j}}{\maclH} =0.
$$
Then
$$
\nmm {f_{0,j}-f_{0,k}} \le C_0\nm {f_{0,j}-f_{0,k}}{\maclH} \to 0,
\quad \text{as}\quad
j,k\to \infty .
$$
Since $\left |  \nmm {f_{0,j}}-\nmm {f_{0,k}}  \right | \le \nmm {f_{0,j}-f_{0,k}}$, it follows
that $\{ \nmm {f_{0,j}} \} _{j=1}^\infty$ is a Cauchy sequence in $\mathbf R$. Hence
$$
\nmm f \equiv \lim _{j\to \infty} \nmm {f_{0,j}}
$$
exists and is independent of the chosen particular sequence $\{ f_{0,j}\} _{j=1}^\infty$.
Since \eqref{Eq:NormEquiv3} holds for any $f_{0,j}\in \mho$, it follows from
the recent estimates and limit properties that \eqref{Eq:NormEquiv3}
extends to any $f\in \maclH$. This gives (1).

\par

%
%

\par

For $f,g\in \maclH$, their scalar product $(f,g)_{[n_k]}$ can be evaluated by
$$
(f,g)_{[n_k]} =\frac 14 \left (
\nm {f+g}{[n_k]}^2 - \nm {f-g}{[n_k]}^2 + i\nm {f+ig}{[n_k]}^2
- i\nm {f-ig}{[n_k]}^2.
\right )
$$
By letting $k$ tends to $\infty$, it follows that $\nmm \cdo$ is a Hilbert
norm with scalar product
$$
(f,g) \equiv \frac 14 \left (
\nmm {f+g}^2 - \nmm {f-g}^2 + i\nmm {f+ig}^2
- i\nmm {f-ig}^2 \right ),
$$
giving that $\nmm \cdo$ fulfills (2).

\par

It remains to prove that (3) holds. By repetition it suffices to prove
\begin{equation}\label{Eq:TranslModNormInv}
\nmm {e^{i\scal \cdo \xi}f(\cdo -x)}= \nmm f
\end{equation}
when $x=0$ and $\xi = \xi _je_j$ or when $\xi =0$ and $x=x_je_j$
for some $x_j,\xi _j\in \mathbf R$. Here $e_j$ denotes the vector
of order $j$ in the standard basis of $\rr d$. Then we only prove
\eqref{Eq:TranslModNormInv} for $\xi =0$ and $x=x_je_j$. The
other cases follow by similar arguments and are left for the reader.

\par

Let $R$ be chosen such that $R>|x_j|$. Then $Q_R$ and $-x_je_j+Q_R$ intersects.
We have
\begin{align}
\nm {f(\cdo -x_je_j)}{[R]}^2
&=
\frac 1{|Q_R|}\iint _{Q_R} \nm {e^{i\scal \cdo {\eta }}f(\cdo -(x_je_j+y))}{\maclH}^2\, dyd\eta 
\notag
\\[1ex]
&=
\frac 1{|Q_R|}\iint _{-x_je_j+Q_R} \nm {e^{i\scal \cdo {\eta}}f(\cdo -y)}{\maclH}^2\, dyd\eta 
\notag
\\[1ex]
&=
\frac 1{|Q_R|}\iint _{Q_R} \nm {e^{i\scal \cdo {\eta}}f(\cdo -y)}{\maclH}^2\, dyd\eta 
+E_R(f)
\notag
\\[1ex]
&=
\nm f{[R]}^2
+E_R(f),
\label{Eq:NormReform}
\end{align}
where
\begin{multline*}
E_R(f) = \frac 1{|Q_R|}\iint _{-x_je_j+Q_R} \nm {e^{i\scal \cdo {\eta}}f(\cdo -y)}{\maclH}^2\, dyd\eta 
\\
-
\frac 1{|Q_R|}\iint _{Q_R} \nm {e^{i\scal \cdo {\eta}}f(\cdo -y)}{\maclH}^2\, dyd\eta .
\end{multline*}
If
$$
\Delta _R = \big ( (-x_je_j+Q_R)\setminus Q_R \big )
\bigcup \big ( Q_R \setminus (-x_je_j+Q_R) \big ),
$$
then it follows that $|\Delta _R| = 2|x_j|(2R)^{2d-1}$ and that
$$
|E_R(f)| \le \frac 1{|Q_R|}
\iint _{\Delta _R} \nm {e^{i\scal \cdo {\eta}}f(\cdo -y)}{\maclH}^2\, dyd\eta .
$$
This gives
\begin{align*}
|E_R(f)| &\le \frac 1{|Q_R|}\iint _{\Delta _R} \nm {e^{i\scal \cdo {\eta}}f(\cdo -y)}{\maclH}^2\, dyd\eta 
\\[1ex]
&\le \frac {C_0}{|Q_R|}\iint _{\Delta _R} \nm f{\maclH}^2\, dyd\eta 
\\[1ex]
&=
\frac {C_0|\Delta _R|\nm f{\maclH}^2}{|Q_R|} = \frac {C_0|x_j|\nm f{\maclH}^2}{R},
\end{align*}
which tends to zero as $R$ turns to infinity.

\par

By letting $R=n_k$ in the previous analysis, \eqref{Eq:NormReform} gives
\begin{align*}
\nmm {f(\cdo -x_je_j)} &= \lim _{k\to \infty} \nm {f(\cdo -x_je_j)}{[n_k]}^2
\\[1ex]
&=
\lim _{k\to \infty} \left (
\frac 1{|Q_{n_k}|}\iint _{Q_{n_k}} \nm {e^{i\scal \cdo {\eta}}f(\cdo -y)}{\maclH}^2\, dyd\eta 
+E_{n_k}(f) \right )
\\[1ex]
&=
\nmm {f} +0,
\end{align*}
which gives (3) and thereby the result.
\end{proof}

\par

\begin{proof}[Proof of Theorem \ref{Thm:InvHilbertSpaces2}]
By Lemma
\ref{Lemma:ConvEquivHilbNorm}, we may replace
the norm for $\maclH$ by an equivalent norm which is
invariant under translations and modulations. The result now
follows from Theorem \ref{Thm:InvHilbertSpaces1}.
\end{proof}

\par

\section{Separability of Hilbert spaces embedded in distribution
spaces}\label{sec2}

\par

In this section we give a proof of that Hilbert spaces which are
continuously embedded in suitable distribution spaces are separable.
We remark that these distribution spaces can be significantly larger than
the set of tempered distributions, $\mascS '(\rr d)$.
Especially we deduce a generalization of Proposition \ref{Prop:Separability}
(see Proposition \ref{Prop:Separability2} below).

\medspace

First we introduce some test function spaces and their distribution spaces,
which are under consideration.
We recall that
the Pilipovi{\'c} space $\maclH _\flat (\rr d)$ consists of
all Hermite function expansions
\begin{equation}\label{Eq:HermFuncExp}
f(x) = \sum _{\alpha \in \nn d} c(f,\alpha )h_\alpha (x),
\end{equation}
where the Hermite coefficients $c(f,\alpha )$ should satisfy
$$
|c(f,\alpha )| \lesssim h^{|\alpha |}\alpha !^{-\frac 12},
$$
for some $h>0$. As in \cite{Toft18}, we equip $\maclH _\flat (\rr d)$
with the inductive limit topology of $\maclH _{\flat ,h} (\rr d)$ with
respect to $h>0$. Here $\maclH _{\flat ,h} (\rr d)$
is the Banach space of all smooth functions $f$ on $\rr d$
such that
$$
\nm f{\maclH _{\flat ,h}}
\equiv
\sup _{\alpha \in \nn d}
\left (
\frac {|c(f,\alpha )|\alpha !^{\frac 12}}{h^{|\alpha |}}
\right )
$$
is finite.
%
%

\par

The distribution space
$\maclH _\flat '(\rr d)$ can be identified with the set
of all formal expansions in \eqref{Eq:HermFuncExp}
such that
$$
|c(f,\alpha )| \lesssim h^{|\alpha |}\alpha !^{\frac 12},
$$
for every $h>0$. The topology of  $\maclH _\flat '(\rr d)$
is the projective limit topology of $\maclH _{\flat ,h}'(\rr d)$ with
respect to $h>0$. Here $\maclH _{\flat ,h}'(\rr d)$
is the Banach space of all smooth functions $f$ on $\rr d$
such that
$$
\nm f{\maclH _{\flat ,h}'}
\equiv
\sup _{\alpha \in \nn d}
\left (
\frac {|c(f,\alpha )|\alpha !^{-\frac 12}}{h^{|\alpha |}}
\right )
$$
is finite. It follows that $\maclH _\flat (\rr d)$ is complete and
that $\maclH _\flat '(\rr d)$ is a Fr{\'e}chet space.

\par

The distribution action between
$\maclH _\flat (\rr d)$ and $\maclH _\flat '(\rr d)$
is then given by
\begin{equation}
\label{Eq:FuncDistHermite}
(f,\phi ) _{L^2}
=
\sum _{\alpha \in \nn d} c(f,\alpha )\overline {c(\phi ,\alpha )}.
\end{equation}

\par

Next we recall some facts concerning the Gelfand-Shilov space $\Sigma _1(\rr d)$
and its distribution space $\Sigma _1'(\rr d)$, given in e.{\,}g. \cite{Pil,Toft18}.
We recall that $\Sigma _1(\rr d)$ consists of all
$f\in C^\infty (\rr d)$ such that for every $h>0$, there is a constant $C_h>0$
such that
$$
|x^\alpha \partial ^\beta f(x)| \le C_hh^{|\alpha +\beta |}\alpha !\beta !\, .
$$
The smallest choice of $C_h$ defines a semi-norm on $\Sigma _1(\rr d)$,
and by defining a (project limit) topology from these semi-norms, it
follows that $\Sigma _1(\rr d)$ is a Fr{\'e}chet space. 

\par

There are several characterizations of $\Sigma _1(\rr d)$ and its
(strong) dual or distribution space $\Sigma _1'(\rr d)$. For example,
$\Sigma _1(\rr d)$ is the set of all expansions in \eqref{Eq:HermFuncExp}
such that
$$
|c(f,\alpha )| \le C_re^{-r|\alpha |^{\frac 12}}
$$
for every $r>0$. The smallest $C_r$ defines a semi-norm for $\Sigma _1(\rr d)$,
and the (project limit) topology in this setting agrees with the topology for $\Sigma _1(\rr d)$.
We may then identify $\Sigma _1'(\rr d)$ with all expansions in \eqref{Eq:HermFuncExp}
such that
$$
|c(f,\alpha )| \le Ce^{r_0|\alpha |^{\frac 12}},
$$
for some constants $C>0$ and $r_0>0$, with distribution action given by
\eqref{Eq:FuncDistHermite} when $f\in \Sigma _1'(\rr d)$ and $\phi \in \Sigma _1(\rr d)$.

\par

In the same way we may identify $\mascS (\rr d)$ and $\mascS '(\rr d)$ as the sets of all
expansions in \eqref{Eq:HermFuncExp} such that
$$
|c(f,\alpha )| \le C_r(1+|\alpha |)^{-r}
$$
for every $r>0$, and
$$
|c(f,\alpha )| \le C(1+|\alpha |)^{r_0}
$$
for some $C>0$ and $r>0$, respectively (see e.{\,}g. \cite{ReSi}).
The distribution action is given by \eqref{Eq:FuncDistHermite}
when $f\in \mascS '(\rr d)$ and $\phi \in \mascS (\rr d)$.

\par

From these identifications it is evident that
\begin{equation}\label{Eq:TopVecSpacEmb}
\maclH _\flat (\rr d) \subseteq \Sigma _1(\rr d) \subseteq \mascS (\rr d)
\subseteq
\mascS '(\rr d) \subseteq \Sigma _1'(\rr d) \subseteq \maclH _\flat '(\rr d),
\end{equation}
with dense and continuous embeddings. Since
$\overline {h_\alpha}=(-1)^{|\alpha|}h_\alpha$,
the duality between $\maclH _\flat (\rr d)$ and $\maclH _\flat '(\rr d)$
in \eqref{Eq:FuncDistHermite} can also be expressed by
\begin{equation}
\label{Eq:FuncDistHermite2}
\scal f\phi
=
\sum _{\alpha \in \nn d} (-1)^{|\alpha |}c(f,\alpha )c(\phi ,\alpha ),
\end{equation}
and it follows that $\scal f\phi$ agrees with the usual distribution actions
between elements in $\mascS (\rr d)$ and $\mascS '(\rr d)$,
when $\phi \in \maclH _\flat (\rr d)$ and, more restrictive, $f\in \mascS '(\rr d)$.

\par

We are now prepared to state the generalizations of Lemma
\ref{Lemma:NonGaussAssump} and Proposition
\ref{Prop:Separability} in Section \ref{sec1}.

\par

\begin{lemma}\label{Lemma:NonGaussAssump2}
Suppose that $\maclH$ 
is continuously embedded in $\maclH _\flat ' (\rr d)$.
Then there is a constant $C>0$ such that \eqref{Eq:TypeCondGaussian}
holds.
\end{lemma}

\par

\begin{proof}
Note that $p_{\phi}(f)=|( f ,\phi )|$, $f\in \maclH _\flat '(\rr d)$,
is a continuous seminorm on $\maclH _\flat '(\rr d)$. The continuity of the
inclusion mapping $\maclH \to \maclH _\flat ' (\rr d)$  then yields
$p_{\phi}(f)\leq C_{\phi} \|f\|_{\maclH}$ for some $C_{\phi}>0$
and all $f\in \maclH$, which is the same as \eqref{Eq:TypeCondGaussian},
completing the proof.
\end{proof}

\par

We observe that the Hilbert structure of $\maclH$ actually plays
no role in the previous proof and Lemma
\ref{Lemma:NonGaussAssump2} therefore holds if we just
assume that $\maclH$ is a Banach space that is continuously
embedded in $\maclH _\flat ' (\rr d)$.

\par

Our generalization of Proposition \ref{Prop:Separability}
is the following.

\par

\begin{prop}\label{Prop:Separability2}
Suppose that the Hilbert space $\maclH$ 
is continuously embedded in $\maclH _\flat ' (\rr d)$.
Then $\maclH$ is separable.
\end{prop}

\par

Evidently, by \eqref{Eq:TopVecSpacEmb} it follows that
Proposition \ref{Prop:Separability2} is true with
$\Sigma _1'$ in place of $\maclH _\flat '$.

\par

We need some preparations for the proof of Proposition
\ref{Prop:Separability2}. Especially we shall make use of the Bargmann transform,
$\mathfrak V_d$, defined by
\begin{equation*}
(\mathfrak V_df)(z) =\pi ^{-\frac d4}\int _{\rr d}\exp \Big ( -\frac 12(\scal
z z+|y|^2)+2^{1/2}\scal zy \Big )f(y)\, dy,\quad z \in \cc d.
\end{equation*}
We have
$$
(\mathfrak V_df)(z) =\int_{\rr d} \mathfrak A_d(z,y)f(y)\, dy,
\quad z \in \cc d,
$$
or
\begin{equation}\label{bargdistrform}
(\mathfrak V_df)(z) =\scal f{\mathfrak A_d(z,\cdo )},
\quad z \in \cc d,
\end{equation}
where the Bargmann kernel $\mathfrak A_d$ is given by
$$
\mathfrak A_d(z,y)=\pi ^{-\frac d4} \exp \Big ( -\frac 12(\scal
zz+|y|^2)+2^{1/2}\scal zy\Big ), \quad z \in \cc d, y \in \rr d.
$$
(Cf. \cite{Ba1,Ba2}.)
Here
$$
\scal zw = \sum _{j=1}^dz_jw_j\quad \text{and} \quad
(z,w)= \scal z{\overline w}
$$
when
$$
z=(z_1,\dots ,z_d) \in \cc d\quad  \text{and} \quad w=(w_1,\dots ,w_d)\in \cc d.
$$
Otherwise $\scal \cdo \cdo $ denotes the duality between test function
spaces and their corresponding duals, as above, which is clear from the context.

\par

In \cite{Toft18}, the images of the spaces \eqref{Eq:TopVecSpacEmb} under
the Bargmann transform are presented. Let $A(\cc d)$ be the set of entire
functions on $\cc d$, and let
\begin{align*}
\maclA _{0,\infty}'(\cc d )
&\equiv
\bigcup _{r\ge 0}\sets {F\in A(\cc d)}{|F(z)|\le Ce^{\frac 12 |z|^2}(1+|z|)^r
\ \text{for some}\ C,r>0}
\\[1ex]
\maclA _{0,1}'(\cc d )
&\equiv
\bigcup _{r\ge  0}\sets {F\in A(\cc d)}{|F(z)|\le Ce^{\frac 12 |z|^2+r|z|}
\ \text{for some}\ C,r>0}
\intertext{and}
\maclA _{\flat}'(\cc d )
&\equiv
A(\cc d).
\end{align*}
We equip $\maclA _{0,\infty}'(\cc d )$ and $\maclA _{0,1}'(\cc d )$ with
inductive limit topologies through the semi-norms
\begin{alignat*}{2}
\nm F{\maclA _{0,\infty}'\, ,r} &\equiv \nm {F\cdot e^{-\frac 12 |\cdo |^2}(1+|\cdo |)^{-r}}{L^\infty},
& \quad r&>0,
\intertext{and}
\nm F{\maclA _{0,1}'\, ,r} &\equiv \nm {F\cdot e^{-(\frac 12 |z|^2+r|\cdo |)}}{L^\infty},
& \quad r&>0,
\intertext{respectively. We also equipp $\maclA _{\flat}'(\cc d ) = A(\cc d)$ with the projective
limit topology, given by the semi-norms}
\nm F{\maclA _{\flat}'\, ,r} &\equiv \nm F{B_r(0)}, &
\quad r &> 0.
\end{alignat*}

\par

In \cite{Toft18} it is proved that the Bargmann mappings
\begin{alignat}{2}
\mathfrak V_d &:& \, \mascS '(\rr d) &\to \maclA _{0,\infty}'(\cc d ),
\label{Eq:TempDistBarg}
\\[1ex]
\mathfrak V_d &:& \Sigma _1 '(\rr d) &\to \maclA _{0,1}'(\cc d )
\label{Eq:GeShDistBarg}
\intertext{and}
\mathfrak V_d &:& \maclH _\flat '(\rr d) &\to \maclA _\flat '(\cc d )
\label{Eq:PilDistBarg}
\end{alignat}
are homeomorphisms.

\par

%
%

\par

\begin{proof}[Proof of Proposition \ref{Prop:Separability2}]
We consider the Hilbert space of entire functions $\tilde{\maclH}=\mathfrak V_d(\maclH)$,
provided with the scalar product
$(\mathfrak V_d f, \mathfrak V_d g)_{\tilde{\maclH}}= (f,g)_{\maclH}$.
It thus suffices to show that $\tilde{\maclH}$ is separable. Let $z\in \cc d$. Reasoning
as in the proof of Lemma \ref{Lemma:NonGaussAssump} with the test function
$\mathfrak A_{d}(z, \: \cdot\:)$ instead of $\phi$, we obtain the existence of a
constant $C_{z}>0$ such that
$$
|\mathfrak V_d f(z)|= |\left\langle f, \mathfrak A_{d}(z, \: \cdot \:)\right \rangle |
\leq C_z \|f\|_{\maclH}= C_z \|\mathfrak V_{d}f\|_{\tilde{\maclH}}, \qquad f\in\maclH,
$$
which shows that $\tilde{\maclH}$ is a reproducing kernel Hilbert space. (See e.{\,}g.
\cite{Aro}.)

\par

Let $K_{z}(w)= K(z,w)$ be its reproducing kernel and fix a
countable and dense subset
$D$ of $\cc d$. We claim that the linear span of $\sets {K_z}{z\in D}$ is a
dense subset of $\maclH$.
Indeed, if $\ell$ is a linear functional on $\tilde \maclH$ which vanishes on
$\sets {K_z}{z\in D}$,
then by Riesz representation theorem, there is a unique $G\in \tilde {\maclH}$
such that $\ell (F)=(F,G)_{\tilde \maclH}$, for every $F\in \tilde{\maclH}$.
Since $\ell$ vanish on $\sets {K_z}{z\in D}$, we have
$$
G(z)= (G, K_z)_{\tilde{\maclH}} = \overline {\ell (K_z)} =0
\quad \text{for each}\ z\in D,
$$
giving that $G$, and thereby $\ell$, are identically zero,
in view of the density of $D$ and the continuity of $G$.
The result now follows by the Hahn-Banach theorem.
\end{proof}

\par

\begin{rem}
Let $d\mu (z)=e^{-|z|^2}d\lambda (z)$, $z\in \cc d$, where $d\lambda (z)$ is the
Lebesgue measure on $\cc d$. Then the Bargmann-Fock space, $A^2(\cc d)$,
consists of all $F\in A(\cc d)$ such that
$$
\nm F{A^2}
\equiv 
\left ( \int _{\cc d}|F(z)|^2\, d\mu (z) \right )^{\frac 12}
$$
is finite. We recall that $A^2(\cc d)$ is a Hilbert space with
scalar product
$$
(F,G)_{A^2} = \int _{\cc d}F(z)\overline {G(z)}\, d\mu (z),
\qquad F,G\in A^2(\cc d).
$$
In \cite{Ba1} it is proved that if $\maclH$ and its norm, are
equal to $L^2(\rr d)$ and its norm, then $\tilde \maclH$
and its norm, are equal to $A^2(\cc d)$ and its norm.

\par

Let
$$
\maclA _{\flat}(\cc d)
\equiv
\bigcup _{r\ge 0}\sets {F\in A(\cc d)}{|F(z)|\le Ce^{r|z|}
\ \text{for some}\ C>0},
$$
equipped with the inductive limit topology through the semi-norms
$$
\nm F{\maclA _\flat \, ,r}
\equiv
\nm {F\cdot e^{-r|\cdo |}}{L^\infty}.
$$
Evidently, $\maclA _{\flat}(\cc d)$ is continuously embedded in $A^2(\cc d)$.
In \cite{Toft18} it is proved that
\begin{enumerate}
\item the Bargmann transform is a homeomorphism
from $\maclH _\flat (\rr d)$ to $\maclA _\flat (\cc d)$;

\vrum

\item $\maclA _{\flat}(\cc d) \subseteq A^2(\cc d)\subseteq \maclA _\flat '(\cc d)$
with dense embeddings;

\vrum

\item the map $(F,G)\mapsto (F,G)_{A^2}$
from $\maclA _\flat (\cc d)\times \maclA _\flat (\cc d)$ to $\mathbf C$
extends uniquely to a continuous map from
$\maclA _\flat '(\cc d)\times \maclA _\flat (\cc d)$ or from
$\maclA _\flat (\cc d)\times \maclA _\flat '(\cc d)$ to $\mathbf C$;

\vrum

\item the (strong) dual of $\maclA _{\flat}(\cc d)$ can be
identified with $\maclA _\flat '(\cc d)$, through the form $(\cdo ,\cdo )_{A^2}$.
Furthermore,
$$
(\mathfrak V_df,\mathfrak V_dg)_{A^2}
=
(f,g)_{L^2},
\qquad
f\in \maclH _\flat '(\rr d),\ g\in \maclH _\flat (\rr d).
$$
\end{enumerate}

\par

The reproducing kernel of $A^2(\cc d)$ is given by $K_z(w)=e^{(w,z)}$
(see e.{\,}g. \cite{Ba1}). Hence
\begin{equation}\label{Eq:BargmannRepKernel}
F(z) = (F,e^{(\cdo ,z)})_{A^2},
\end{equation}
when $F\in A^2(\cc d)$. We observe that $K_z\in \maclA _\flat (\cc d)$, for every $z\in \cc d$.
Hence, the right-hand side of \eqref{Eq:BargmannRepKernel} makes sense for any
$F\in \maclA _\flat '(\cc d)$, and because $A^2(\cc d)$ is dense in $F\in \maclA _\flat '(\cc d)$,
it follows that the identity \eqref{Eq:BargmannRepKernel} still holds for any $F\in \maclA _\flat '(\cc d)$.

\par

We observe that the reproducing kernel in the proof of Proposition
\ref{Prop:Separability2} is chosen with respect to the scalar product $(\cdo ,\cdo )_{\tilde \maclH}$,
while the reproducing kernel in \eqref{Eq:BargmannRepKernel} is defined with respect to
$(\cdo ,\cdo )_{A^2}$. It follows that these kernels are not the same, when
$\maclH$ differs from $L^2(\rr d)$.
\end{rem}

\par

\section{Extensions and variations}\label{sec3}

\par

In this section we slightly improve Theorem
\ref{Thm:InvHilbertSpaces2}, and show that
the result is still true when $\maclH$ is
continuously embedded in certain distribution
spaces and ultra-distribution spaces, containing
$\mascS '(\rr d)$. In the first part we recall the
definition of the Pilipovi{\'c} ultra-distribution space
$\maclH _\flat '(\rr d)$ (also denoted by $\maclH _{\flat _1} '(\rr d)$
in \cite{Toft18}). Thereafter we extend
Theorem \ref{Thm:InvHilbertSpaces2}.

\par


\par

We have now the following extension of
Theorem \ref{Thm:InvHilbertSpaces2}.

\par

\begin{thm}\label{Thm:InvHilbertSpaces3}
Let $\maclH$ be a Hilbert space which is
continuously embedded in $\mascD '(\rr d)$ or in
$\maclH _\flat '(\rr d)$,
and such that \eqref{Eq:ContTransMod} holds true for some constant
$C_0\ge 1$ which is independent of $f\in \maclH$ and $x,\xi \in \rr d$.
If $\maclH$ is non-trivial, then $\maclH = L^2(\rr d)$, and
\eqref{Eq:NormEquiv2} holds
for some constant $C>0$ which is independent
of $f\in \maclH = L^2(\rr d)$.
\end{thm}

\par

The proof of Theorem \ref{Thm:InvHilbertSpaces3} 
follows by
a combination of
\cite[Proposition 1.5]{ToGuMaRa},
Theorem \ref{Thm:InvHilbertSpaces2},
and Proposition \ref{Prop:TransModInvQBanachSp} 
below.

\par

We need some preparations for the
proof of Theorem \ref{Thm:InvHilbertSpaces3}.
Since translation and modulation invariant
Hilbert spaces are in focus, we here notice that 
all the spaces in Theorem 
\ref{Thm:InvHilbertSpaces3}
are invariant under such actions,
which is explained in the next result.

\par

\begin{prop}\label{Prop:TranslModTopSpaces}
Let $x,\xi \in \rr d$. Then the map
$f\mapsto f(\cdo -x)e^{i\scal \cdo \xi}$
is a homeo\-morphism on each one of
the spaces in \eqref{Eq:TopVecSpacEmb}
and in $\mascD '(\rr d)$.
\end{prop}

\par

\begin{proof}
The result is evidently true for the
spaces
$$
\mascS (\rr d),\quad
\Sigma _1(\rr d),\quad
\mascS '(\rr d),\quad
\Sigma _1'(\rr d)
\quad \text{and}\quad
\mascD '(\rr d).
$$
We need to prove the result for
$\maclH _\flat (\rr d)$ and
$\maclH _\flat '(\rr d)$.

\par

We recall that in terms of the Weyl map
\begin{equation}\label{Eq:WeylMap}
(W_wF)(z)\equiv e^{-\frac 12|w|^2+(z,w)}F(z-w),
\quad F\in A(\cc d),\ z,w\in \cc d,
\end{equation}
the Bargmann transform
transfer translations and modulations as
\begin{equation}\label{Eq:BargTranslMod}
\begin{aligned}
(\mathfrak V_d(e^{i\scal \cdo \xi }
f(\cdo -x)))(z)
&=
e^{\frac i2\scal x\xi} (W_{\overline w/\sqrt 2}
(\mathfrak V_df))(z),
\\[1ex]
z &\in \cc d,\ w=x+i\xi \in \cc d. 
\end{aligned}
\end{equation}
Since it is evident that the Weyl map
is a homeomorphism on $\maclA _\flat '(\cc d)$
the assertion for $\maclH _\flat '(\rr d)$
is a consequence of the homeomorphism 
\eqref{Eq:PilDistBarg}.
By duality it now follows that translations
and modulations are homeomorphisms on
$\maclH _\flat (\rr d)$ as well, and
the result follows.
\end{proof}

\par

Let $\maclB$ be a vector space.
A \emph{quasi-norm} on $\maclB$ is a 
map from $\maclB$ to $\mathbf R_+\cup \{ 0\}$ such that
\begin{alignat*}{2}
\nm {\alpha f}{\maclB} &= |\alpha |\cdot \nm f{\maclB}, &
\qquad
f &\in \maclB ,\ \alpha \in \mathbf C
\intertext{and}
\nm {f+g}{\maclB}^p &\le  \nm f{\maclB}^p+\nm g{\maclB}^p, &
\qquad
f,g&\in \maclB .
\end{alignat*}
If the topology on $\maclB$ is defined through this quasi-norm,
then $\maclB$ is called a quasi-normed space.
%
%
%
%
%
A complete quasi-normed space is called a \emph{quasi-Banach space}.

\par

%
%

\par

\begin{prop}\label{Prop:TransModInvQBanachSp}
Let $\maclB$ be a quasi-Banach space which is continuously
embedded in $\maclH _\flat '(\rr d)$, and such that for some
constant $C>0$ and function $v\in L^\infty _{loc}(\rr {2d};\mathbf R_+)$,
it holds
\begin{equation}\label{Eq:ExtTranslModInv}
\nm {e^{i\scal \cdo \xi}f(\cdo -x)}{\maclB}
\le
Cv(x,\xi )\nm f{\maclB},
\quad
x,\xi \in \rr d.
\end{equation}
Then the following is true:
\begin{enumerate}
\item $\maclB$ is continuously embedded in $\Sigma _1'(\rr d)$;

\vrum

\item if in addition $v(x,\xi )\le C_0(1+|x|+|\xi |)^N$, for some constants
$C_0>0$ and $N>0$, then $\maclB$ is continuously embedded in
$\mascS '(\rr d)$.
\end{enumerate}
\end{prop}

\par

\begin{proof}
We may assume that $C\ge 1$. Let
$$
v_0(x,\xi )\equiv \sup _{f\in \maclB \setminus 0}
\left (
\frac {\nm {e^{i\scal \cdo \xi}f(\cdo -x)}{\maclB}}{\nm f{\maclB}}
\right ) .
$$
Then $v_0\le Cv$, and $v_0$ is submultiplicative, i.{\,}e.,
\begin{equation}\label{Eq:ModerFunc}
v_0(x+y,\xi +\eta )\le v_0(x,\xi )v_0(y,\eta ),
\qquad
x,y,\xi ,\eta \in \rr d.
\end{equation}
Furthermore, \eqref{Eq:ExtTranslModInv} gives
$$
\nm {e^{i\scal \cdo \xi}f(\cdo -x)}{\maclB}
\le
v_0(x,\xi )\nm f{\maclB},
\quad
x,\xi \in \rr d.
$$
Hence we may assume that \eqref{Eq:ExtTranslModInv} holds
with $v_0$ in place of $v$, and that $C=1$. We also observe that
\eqref{Eq:ModerFunc} gives
\begin{equation}\label{Eq:Submult}
v_0(x) \le Ce^{r(|x|+|\xi |)},
\quad
x,\xi \in \rr d,
\end{equation}
for some constants $C>0$ and $r>0$ (see \cite{Gc2}).

\par

Suppose that $f\in \maclB$ and let $F=\mathfrak V_df$. By the assumptions,
\eqref{Eq:BargTranslMod} and that
\eqref{Eq:PilDistBarg} is a homeomorphism, it follows that for some $C>0$ we have
$$
|F(z-w)e^{-\frac 12|w|^2+\RE (z,w)}| v_0(\sqrt 2\overline w)^{-1}\le C,
\qquad
z\in B_r(0),\ w\in \cc d,
$$
which implies
\begin{equation}\label{Eq:AnalFuncSupEst}
\sup _{z\in \cc d} \left (
|F(z)|e^{-\frac 12|z|^2}v_0(-\sqrt 2 \overline z)^{-1}
\right )<\infty .
\end{equation}
By combining the latter estimate with \eqref{Eq:Submult} shows that
$F\in \maclA _{0,1}'(\cc d)$. The asserted embedding in (1) now follows from the homeomorphic
property of \eqref{Eq:GeShDistBarg}.

\par

The assertion (2) can be found in
e.{\,}g. \cite{GroZim}.
Here we give alternative arguments. 
Therefore, suppose that $v$ satisfies the 
estimates in (2). Then
\eqref{Eq:AnalFuncSupEst} shows that $F\in \maclA _{0,\infty}'(\cc d)$.
 The asserted embedding in (2) now follows from the homeomorphic
property of \eqref{Eq:TempDistBarg},
giving the result.
\end{proof}

\par

\begin{rem}
By the definition of modulation spaces and their mapping properties under the
Bargmann transform, it follows from the estimate 
\eqref{Eq:AnalFuncSupEst} that $\maclB$ in Proposition
\ref{Prop:TransModInvQBanachSp} (1) is continuously embedded in
the modulation space $M^\infty _{(\omega )}(\rr d)$,
with $\omega (x,\xi )=v_0(-x-i\xi )^{-1}$, $x,\xi \in \rr d$. (See e.{\,}g. \cite{Toft18}.)
\end{rem}

\par

\begin{proof}[Proof of Theorem \ref{Thm:InvHilbertSpaces3}]
If $\maclH$ is continuously embedded in $\mascD '(\rr d)$
or in $\maclH _\flat '(\rr d)$,
then \cite[Proposition 1.5]{ToGuMaRa} and
Proposition
\ref{Prop:TransModInvQBanachSp} (2) show that
$\maclH$ is continuously embedded in $\mascS '(\rr d)$.
The result now follows from
Theorem \ref{Thm:InvHilbertSpaces2}.
%
%
\end{proof}

\par


\begin{thebibliography}{99}
\bibitem{Aro}
N. Aronszajn
\emph{Theory of reproducing kernels},
Trans. Amer. Math. Soc. \textbf{68} (1950),
337--404.

\bibitem{BaMoVe}
S. R. Bais, P. Mohan, D. V. Naidu
\emph{A characterization of translation and 
modulation invariant
Hilbert space of tempered distributions},
Arch. Math. \textbf{122} (2024), 429--436.

\bibitem{Ba1} V. Bargmann
\emph{On a Hilbert space of analytic
functions and an associated integral transform}, 
Comm. Pure
Appl. Math., \textbf{14} (1961), 187--214.

\bibitem{Ba2}
V. Bargmann
\emph{On a Hilbert space of analytic functions 
and an associated
integral transform. Part II. A family of related
function spaces. Application to
distribution theory}, Comm. Pure
Appl. Math. \textbf{20} (1967), 1--101.

\bibitem{Fe1}
H. G. Feichtinger
\emph{Modulation spaces on locally
compact abelian groups. Technical report}, 
{University of
Vienna}, Vienna, 1983; also in:
M. Krishna, R. Radha,
S. Thangavelu (Eds)
Wavelets and their applications, Allied
Publishers Private Limited, NewDelhi Mumbai 
Kolkata Chennai Nagpur
Ahmedabad Bangalore Hyderabad Lucknow, 2003,
pp. 99--140.

\bibitem{Fei6}
H. G. Feichtinger
\emph{Modulation spaces: Looking back and ahead},
Sampl. Theory Signal Image Process.
\textbf{5} (2006), 109--140.

\bibitem{FeLuCo}
H. Feichtinger, F. Luef, E. Cordero
\emph{Banach Gelfand triples for Gabor 
analysis}, in: L. Rodino, M. W. Wong (Eds),
Pseudo-Differential Operators, Quantization and 
Signals, Lecture notes
in math. \textbf{1949}, Springer, Berlin 
Heidelberg, 2008, pp. 1--33.

\bibitem{Gc1}
K. Gr{\"o}chenig
\emph{Foundations of
Time-Frequency Analysis},
\newblock Birkh{\"a}user, Boston, 2001.

\bibitem{Gc2}
K. Gr{\"o}chenig
\emph{Weight functions in time-frequency 
analysis,
\rm {in: L. Rodino, M. W. Wong (Eds.) 
Pseudodifferential
Operators: Partial Differential Equations and 
Time-Frequency Analysis}},
Fields Institute Comm., \textbf{52} (2007),
pp. 343--366.

\bibitem{GroZim}
K. Gr{\"o}chenig, G. Zimmermann
\emph{Spaces of test functions via the
STFT},
J. Funct. Spaces Appl. \textbf{2} (2004),
25--53.

\bibitem{Ho1}
L. H{\"o}rmander
\emph{The Analysis of Linear
Partial Differential Operators}, vol {I},
Springer-Verlag, Berlin Heidelberg NewYork 
Tokyo, 1983.

\bibitem{Pil}
S. Pilipovi{\'c}
\emph{Tempered ultradistributions},
Boll. U.M.I. \textbf{7} (1988), 235--251.

\bibitem{ReSi}
M. Reed and B. Simon
\emph{Methods of modern
mathematical physics}, Academic Press, London 
New York, 1979.

\bibitem{Rud}
W. Rudin
\emph{Functional analysis},
McGraw-Hill, New York, D{\"u}sseldorf,
Johannesburg, 1973.

\bibitem{Tan} S. Thangavelu
\emph{Lectures on Hermite and Laguerre 
expansions},
Mathematical Notes \textbf{42},
Princeton University Press,
Princeton, N.J., 1993.

\bibitem{Toft18}
J. Toft
\emph{Images of function and distribution
spaces under the
Bargmann transform},
J. Pseudo-Differ. Oper. Appl. \textbf{8}
(2017), 83--139.

\bibitem{Toft20}
J. Toft
\emph{Schatten properties, nuclearity and 
minimality of phase
shift invariant spaces}, Appl. Comput. Harmon. 
Anal. \textbf{46} (2019),
154--176.

\bibitem{ToGuMaRa}
J. Toft, A. Gumber, R. Manna, P. K. Ratnakumar
\emph{Translation and modulation invariant 
Hilbert spaces},
Monatsh. Math. \textbf{196} (2021), 389--398.
\end{thebibliography}
\end{document}